%% file: main.tex
\documentclass[letterpaper, 10 pt, conference]{ieeeconf}  

\IEEEoverridecommandlockouts                              
\usepackage{url}
\usepackage[pdftex]{graphicx}
\usepackage{tikz,pgfplots}
\usetikzlibrary{positioning}
\pgfplotsset{compat=1.14} 
\tikzstyle{vertex} = [fill,shape=circle,node distance=30pt]
\tikzstyle{edge} = [fill,opacity=.6,fill opacity=.5,line cap=round, line join=round, line width=10pt]
\tikzstyle{elabel} =  [fill,shape=circle,node distance=30pt,fill opacity=.9]
\definecolor{mygray}{gray}{0.95}
\pgfdeclarelayer{background}
\pgfsetlayers{background,main}
\usepackage[most]{tcolorbox}
\definecolor{mypurple}{rgb}{0.59, 0.44, 0.84}
\usepackage{hyperref}

\usepackage[utf8]{inputenc}
\usepackage[pdftex]{graphicx}
\usepackage{amsmath,amssymb}
\usepackage{color}
\usepackage{mathrsfs}
\usepackage{eucal}
\usepackage{mathtools}
\usepackage{nameref}        
\usepackage{algorithmic}
\usepackage{algorithm}
\usepackage{pgfplots}
\pgfplotsset{compat=1.16}

\newtheorem{thm}{Theorem}

\newtheorem{remark}{Remark}
\newtheorem{defn}{Definition}

\newcommand{\R}{\mathbb{R}}

\newcommand{\Z}{\mathbb{Z}}

\newcommand{\g}{\mathcal{G}}

\newcommand{\Vs}{\mathcal{V}}
\newcommand{\Es}{\mathcal{E}}
\newcommand{\xv}{\mathbf{x}}
\newcommand{\uv}{\mathbf{u}}

\newcommand{\Av}{\mathbf{A}}
\newcommand{\Cv}{\mathbf{C}}
\newcommand{\Bv}{\mathbf{B}}
\newcommand{\Dv}{\mathbf{D}}
\newcommand{\Jv}{\mathbf{J}}

\newcommand{\tA}{\textsf{A}}
\newcommand{\tT}{\textsf{T}}

\newcommand{\yv}{\mathbf{y}}
\newcommand{\hv}{\mathbf{h}}
\newcommand{\fv}{\mathbf{f}}
\newcommand{\gv}{\mathbf{g}}
\newcommand{\Id}{\mathbf{I}}
\newcommand{\tra}{\prime}
\newcommand{\h}{\mathcal{H}}

\title{\LARGE \bf Observability of Hypergraphs}

\author{Joshua Pickard,
\thanks{J. Pickard is with the Department of  Computational Medicine \& Bioinformatics, Medical School, University of Michigan, Ann Arbor, MI 48109 USA (e-mail: jpic@umich.edu).} Amit Surana,
\thanks{A. Surana is with Raytheon Technologies Research Center, East Hartford, CT 06108 USA (e-mail: amit.surana@rtx.com).}
Anthony Bloch, \thanks{A. Bloch is with the Department of Mathematics, University of Michigan, Ann Arbor, MI 48109 USA (e-mail: abloch@umich.edu).}
and Indika Rajapakse
\thanks{I. Rajapakse is with the Department of Computational Medicine \& Bioinformatics, Medical School and the Department of Mathematics, University of Michigan, Ann Arbor, MI 48109 USA (e-mail: indikar@umich.edu).}
\thanks{All code and data associated with this paper are available at \href{https://github.com/Jpickard1/observability-of-hypergraphs}{https://github.com/Jpickard1/observability-of-hypergraphs}}
}

\begin{document}
\maketitle
\thispagestyle{empty}
\pagestyle{empty}

\begin{abstract}
In this paper we develop a framework to study observability for uniform hypergraphs. 
Hypergraphs, being extensions of graphs, allow edges to connect multiple nodes and unambiguously represent multi-way relationships which are ubiquitous in many real-world networks. We extend the canonical homogeneous polynomial or multilinear dynamical system on uniform hypergraphs to include linear outputs, and we derive a Kalman-rank-like condition for assessing the local weak observability. We propose an exact techniques for determining the local observability criterion, and we propose a greedy heuristic to determine the minimum set of observable nodes. Numerical experiments demonstrate our approach on several hypergraph topologies and a hypergraph representations of neural networks within the mouse hypothalamus.
\end{abstract}

\section{INTRODUCTION}
The ability to monitor, predict, and control complex, networked systems is a fundamental and crucial task with widespread applications in various domains, including social/communications systems, life sciences, security/defense, and more \cite{strogatz2001exploring,newman2018networks,bullo2020lectures}. Networks are often represented as graphs, which while simple and to some extent universal, only represent pairwise relationships, whereas real-world phenomena can be rich in multi-way relationships. Examples include social networks with friend groups, the colocalization of chromatin strands to form transcription clusters, and brain activity where multiple regions are coregulated \cite{wolf2016advantages,battiston2020networks,benson2021higher}.
In each case, observing the behavior and state of a few key elements within the system is informative for the global state of the system, and as opposed to graphs, hypergraphs provide a more precise representation of the system structure.

A hypergraph is a generalized form of a graph, where its hyperedges can connect any number of vertices, explicitly capturing multi-way relationships \cite{berge1984hypergraphs}. Tensors offer a natural framework for representing multi-dimensional patterns and capturing higher-order interactions \cite{kolda2009tensor}, making them increasingly relevant in the study of hypergraphs \cite{chen2021controllability,9119161,surana2022hypergraph}.


Observability in dynamical systems quantifies our capability to deduce the system's internal states from a given set of system outputs or measurements. For instance, in control engineering, especially when designing feedback control systems, we rely on estimations of the plant state based solely on the plant output or the measurements collected from its sensors. This finds various applications, such as monitoring chemical reactions network or understanding the spread of information or a disease within a community. In the context of networked systems, two fundamental questions arise:
\begin{itemize}
    \item (Q1) Is a set of sensor nodes sufficient to render a network observable?
    \item (Q2) What is the minimum set of nodes to render a network observable?
\end{itemize}

Observability of network systems has been extensively studied from several perspectives; see \cite{montanari2020observability} and references therein. Structural observability involves determining Q1 based on methods such as the underlying directed graph structure; dynamic observability addresses Q2 based on classical matrix properties, particle filtering \cite{montanari2019particle}, or the observability gramian \cite{summers2014optimal}; and topological observability explores the relationship between observability and graph topologies \cite{liu2013observability,su2017analysis}. While hypergraphs are finding increasing use in representing complex networks, the problem of hypergraph observability remains unexplored.

This paper contributes to the observability of hypergraph dynamics in the following ways:
\begin{itemize}
    \item We construct a nonlinear observability test for hypergraph dynamics with linear outputs to answer Q1.
    \item We propose a greedy algorithm to efficiently determine the minimum set of observable nodes (MON), in response to Q2.
    \item We demonstrate our approach on several uniform hypergraph topologies and hypergraphs derived from an experimental mouse endomicroscopy dataset.
\end{itemize}
In this paper, we focus on the concept of weak local observability for nonlinear systems. To overcome the limitations of local observability, our proposed algorithms leverage symbolic calculations to offer a global observability test.

This paper is organized as follows. Following preliminaries in Section \ref{sec: prelims}, Section \ref{sec:obs} introduces nonlinear observability, and Section \ref{sec: uniform hypergraphs} provides an overview of uniform hypergraphs and their dynamics. In Section \ref{sec: obsv uniform hypergraphs}, a test for hypergraph observability is proposed and Section \ref{sec: MONS} provides our algorithm for selecting the MON. Finally, numerical results and a discussion are provided in Sections \ref{sec: numerical results} and \ref{sec: conclusion} respectively.


\section{PRELIMINARIES}\label{sec: prelims}
In this section, we present a concise review of the multi-linear algebra and Lie theory necessary for the development of a hypergraph observability criteria.

\subsection{Kronecker Product}
The Kronecker product of  $\Av\in\R^{m\times n}$ and $\mathbf{B}\in\R^{p\times q}$ is given by,
\begin{equation*}\label{eq:kronM}
\Av\otimes\Bv=\left(
                                 \begin{array}{ccc}
                                   \Av_{11}\Bv & \cdots & \Av_{1n}\Bv \\
                                   \vdots & \ddots & \vdots \\
                                    \Av_{m1}\Bv & \cdots & \Av_{mn}\Bv \\
                                 \end{array}
                               \right),
\end{equation*}
where, $\Av\otimes\Bv\in \R^{mp\times nq}$. Furthermore,  the mixed product property implies that,
\begin{equation*}\label{eq:kronprod}
(\Av\otimes\Bv)(\Cv\otimes\mathbf{D})=(\Av\Cv)\otimes(\Bv\Dv),
\end{equation*}
where $\Av,\Bv,\Cv,$ and $\Dv$ are matrices of compatible dimensions. The Kronecker power is a convenient notation to express all possible products of elements of a vector up to a given order, and it is denoted by,
\begin{equation*}\label{eq:powx}
\xv^{[i]}=\underbrace{\xv\otimes\xv\cdots\otimes\xv}_{i-\textsf{times}}.
\end{equation*}
Moreover, for $\xv\in\R^n,$ $\mbox{dim}(\xv^{[i]})= n^i$,  and each component of $\xv^{[i]}$ is of the form $\xv_1^{\omega_1}\xv^{\omega_2}\cdots \xv_n^{\omega_n} $ for some multi-index $\mathbf{\omega}\in \Z^n$  of weight $\sum_{j=1}^n\omega_j=i$.

\subsection{Tensors}
A tensor is a multidimensional array \cite{kolda2009tensor, chen2019multilinear,chen2021multilinear}. The order of a tensor is the number of its dimensions, and each dimension is called a mode. An $m$-th order real valued tensor will be denoted by $\tT\in \mathbb{R}^{J_1\times J_2\times  \dots \times J_m}$, where $J_k$ is the size of its $k$-th mode. We will denote by $\mathcal{J}=(J_1,J_2,\cdots,J_m)$. It is therefore reasonable to consider scalars $x\in\mathbb{R}$ as zero-order tensors, vectors $\xv\in\mathbb{R}^{n}$ as first-order tensors, and matrices $\mathbf{X}\in\mathbb{R}^{m\times n}$ as second-order tensors. A tensor is called \textit{cubical} if every mode is the same size, i.e., $\textsf{T}\in \mathbb{R}^{n\times n\times  \dots \times n}$. A cubical tensor $\tT$ is called \textit{supersymmetric} if $\tT_{j_1j_2\dots j_k}$ is invariant under any permutation of the indices.

\begin{defn}
The \textit{tensor vector multiplication} $\tT \times_{p} \textbf{v}$ along mode $p$ for a vector $\textbf{v}\in  \mathbb{R}^{J_p}$ is defined  by
\begin{equation*}
(\tT \times_{p} \textbf{v})_{j_1j_2\dots j_{p-1}j_{p+1}\dots j_k}=\sum_{j_p=1}^{J_p}\tT_{j_1j_2\dots j_p\dots j_k}\textbf{v}_{j_p},
\end{equation*}
which can be extended to
\begin{equation}\label{eq6}
\begin{split}
\tT\times_1 \textbf{v}_1 \times_2\textbf{v}_2\times_3\textbf{v}_3\dots \times_{k}\textbf{v}_k=\tT\textbf{v}_1\textbf{v}_2\textbf{v}_3\dots\textbf{v}_m\in\mathbb{R}
\end{split}
\end{equation}
for $\textbf{v}_p\in \mathbb{R}^{J_p}$. The expression (\ref{eq6}) is also known as the homogeneous polynomial associated with $\tT$. If $\textbf{v}_p=\textbf{v}$ for all $p$, we write  (\ref{eq6})  as $\tT\textbf{v}^{m}$ for simplicity.
\end{defn}

Tensor unfolding is considered as a critical operation in tensor computations \cite{kolda2009tensor}. In order to unfold a tensor $\tT\in\R^{J_1\times J_2\times\cdots \times J_m}$ into a vector or a matrix, we use an index mapping function $ivec(\cdot,\mathcal{J}):\mathbb{Z}\times \mathbb{Z}\times \stackrel{\scriptscriptstyle m}{\cdots} \times \mathbb{Z} \rightarrow \mathbb{Z}$ as defined in \cite{doi:10.1137/110820609}, which is given by
\begin{equation*}
ivec(\textbf{j},\mathcal{J}) = j_1+\sum_{k=2}^m(j_k-1)\prod_{l=1}^{k-1}J_l.
\end{equation*}
where, $\textbf{j}=(j_1,j_2,\cdots,j_m)$.
\begin{defn}
The $k$-\textit{mode unfolding} of $\tT$ denoted by $\tT_{(k)}$, is a $J_k \times (J_1\cdots J_{k-1}J_{k+1}\cdots J_m)$ matrix, whose $(i,p)$-th entries are given by
\begin{equation*}
\tT_{(k)}(i,p)=\tT_{j_1,\cdots,j_{k-1},i,j_{k},\cdots,j_m},
\end{equation*}
where, $\tilde{\textbf{j}}=(j_1,\cdots, j_{k-1},j_{k+1},\cdots,j_m)$ is such that $p=ivec(\tilde{\textbf{j}},\tilde{\mathcal{J})}$ with $\tilde{\mathcal{J}}=(J_1,\cdots,J_{k-1},J_{k+1}\cdots J_m)$.
\end{defn}

\subsection{Lie Derivatives}
Let $h: \R^n\rightarrow \R$ be a scalar function, then its gradient is defined as a row vector of partial derivatives,
\begin{equation*}
dh=\begin{pmatrix}
    \frac{\partial}{\partial x_1}h & \dots & \frac{\partial}{\partial x_n}h
\end{pmatrix}.
\end{equation*}
This definition can be generalized to gradient of a vector valued function $\hv: \R^n\rightarrow \R^m$ with components $\hv=(h_1,\cdots,h_m)^\prime$,   as
\begin{equation*}
\nabla_{\xv}\hv=\left(\begin{array}{c}
                  dh_1 \\
                  \vdots \\
                  dh_m
                \end{array}\right).
\end{equation*}
Let $\langle \cdot,\cdot \rangle$ be the standard inner product on $\R^n$. Let $\fv: \R^n\rightarrow \R^n$ be a vector field, then \textit{Lie derivative} of a scalar function $h$ along $\fv$ is defined as
\begin{equation*}
L_{\fv}h=\langle dh,\fv\rangle.
\end{equation*}
One can generalize this to higher order Lie derivatives $L^i_{\fv}h, i\in\Z$ defined recursively as follows,
\begin{equation*}
L^i_{\fv}h=L_{\fv}(L^{i-1}_{\fv}h),\mbox{with } L^0_{\fv}h=h.
\end{equation*}
For vector valued function $\hv$ one can similarly define the Lie derivative as
\begin{equation*}
L_{\fv}\hv=\left(\begin{array}{c}
                 L_{\fv}h_1 \\
                  \vdots \\
                 L_{\fv}h_m
                \end{array}\right)=
                \left(\begin{array}{c}
                 \langle dh_1,\fv\rangle \\
                  \vdots \\
                  \langle dh_m,\fv\rangle
                \end{array}\right).
\end{equation*}
This definition can be naturally extended to higher order  by applying the definition of higher order Lie derivatives for scalar functions to the components of $\hv$.

\section{NONLINEAR OBSERVABILITY CRITERION}\label{sec:obs}
For nonlinear systems, notions of controllability and observability were introduced in the seminal work \cite{hermankrener}. The notion of observability is based on the indistinguishability of system states, but in contrast to the linear systems, there are several nonlinear observability concepts, such as local, weak and global observability \cite{hermankrener,sontag1984concept,gerbet2020global}. Unfortunately, unlike the linear case where the Kalman rank condition can be used to determine observability, no easy criteria exist for nonlinear systems.

Consider the affine control system $\Sigma$,
\begin{equation*}
\Sigma\begin{cases}
\dot{\xv}=\fv(\xv,\uv)=\hv_0(\xv)+\sum_{i=1}^k\hv_i(\xv)u_i\\
\yv=\gv(\xv)
\end{cases}
\end{equation*}
where, $\uv=(u_1,\cdots,u_k)^\tra\in \R^k$ denotes the input vector, $\xv\in M\subset \R^n$ is the state vector and $\yv\in\R^m$ is the output/measurement vector. We assume that $\Sigma$ is analytic, i.e., the functions $\hv_i:M\rightarrow M, i=0,\cdots,k$ and $g_i:M\rightarrow \R, i=1,\cdots,m$ where $\gv=(g_1,\cdots,g_m)^\tra$ are assumed to be analytic functions defined on $M$. We also have to assume $\Sigma$ is complete, that is, for every bounded measurable input $\uv(t)$ and every $\xv_0\in M$ there exists a solution $\xv(t)$ of $\Sigma$ such that $\xv(0) = \xv_0$ and $\xv(t)\in M$ for all $t\in \R$. We review different notions of observability from \cite{obsThesis} which are equivalent to those introduced in \cite{hermankrener}, but use a slightly different terminology.

\begin{defn}
Let $U$ be an open subset of $M$. A pair of points $\xv_0$ and $\xv_1$ in $M$ are called $U$-\textit{distinguishable } if there exists a measurable input $\uv(t)$ defined on the interval $[0,T]$ that generates solutions $\xv_0(t)$ and $\xv_1(t)$ of system $\Sigma$ satisfying $\xv_i(0)=\xv_i,i=0,1$ such that $\xv_i(t)\in U$ for $t\in[0,T]$ and $\hv(\xv_0(t))\neq \hv(\xv_1(t))$ for some $t\in [0,T]$. We denote by $I(\xv_0,U)$ all points $\xv_1\in U$ that are not $U$--\textit{distinguishable } from $\xv_0$
\end{defn}

\begin{defn}
The system $\Sigma$ is \textit{observable} at $\xv_0\in M$ if $I(\xv_0,M)=\xv_0$.
\end{defn}

\begin{defn}
The system $\Sigma$ is \textit{locally observable} at $\xv_0\in M$ if for every open neighbourhood $U$ of $\xv_0$, $I(\xv_0,U)=\xv_0$
\end{defn}

Local observability implies observability. On the other hand, since $U$ can be chosen arbitrarily small, local observability implies that we can distinguish between neighboring points instantaneously. Both the definitions above ensure that a point $\xv_0\in M$   can be distinguished from every other point in $M$. It is often sufficient to distinguish between neighbours in $M$, which leads to the following two notions of observability.

\begin{defn}
The system $\Sigma$ is \textit{weakly observable} at $\xv_0\in M$ if $\xv_0$ has an open neighbourhood $U$ such that $I(\xv_0,M)\bigcap U=\xv_0$.
\end{defn}

\begin{defn}
The system $\Sigma$ is \textit{locally weakly observable} at $\xv_0\in M$ if $\xv_0$ has an open neighbourhood $U$ such that for every open neighbourhood $V$ of $\xv_0$ contained in $U$, $I(\xv_0,V)=\xv_0$.
\end{defn}

As we can set $U=M$, local observability implies local weak observability. The local weakly observability lends itself to a simple algebraic test.  Let $\h$ be the \textit{observation space},
\begin{eqnarray*}
\h&=&\{L_{\hv_{i_1}}L_{\hv_{i_2}}\cdots L_{\hv_{i_r}}(g_{i}): r\geq 0, i_j=0,\cdots,k,\notag\\
 &&i=1,\cdots,m\}, \label{eq:defG}
\end{eqnarray*}
and
\begin{equation*}\label{eq:defdG}
d\h=\mbox{span}_{\R_x}\{d\phi:\phi \in \h\},
\end{equation*}
be the space spanned by the gradients of the elements of $\h$, where $\R_x$ is space of meromorphic functions on $M$. The following result was proved in \cite{hermankrener}, see Theorems $3.1$ and $3.11$.
\begin{thm}\label{thm1}
The analytic system $\Sigma$ is locally weakly observable for all $\xv$ in an open dense set of $M$ if and only if $dim_{\R_x}(d\h) = n$.
\end{thm}
\begin{remark}
Here $dim_{\R_x}(d\h)$ is the generic or maximal rank of $d\h$, that is, $dim_{\R_x}(d\h)= \max_{\xv\in M}(dim_{\R} d\h(\xv))$.
\end{remark}
For system $\Sigma$ with no control inputs, i.e. $\hv_i\equiv0,i=1,\cdots,k$ the condition for local weak observability simplifies to checking,
\begin{equation}\label{eq:obsr}
\mbox{rank}(\mathcal{O}(\xv))=n,
\end{equation}
where, $\mathcal{O}(\xv)$ is the nonlinear observability matrix (NOM),
\begin{equation}\label{eq:obsrmat}
    \mathcal{O}(\xv)=\nabla_{\xv}\begin{pmatrix}
    L^0_{\hv_0}\gv(\xv)\\
    L^1_{\hv_0}\gv(\xv)\\
    \vdots\\
    L^r_{\hv_0}\gv(\xv)
    \end{pmatrix},
\end{equation}
for some $r\in \Z$. One can use symbolic computation to check the generic rank condition (\ref{eq:obsr}) as performed by Sedoglavic's algorithm \cite{sedoglavic2001probabilistic}.

\begin{remark}\label{remrn}
In general the value of $r$  to use in (\ref{eq:obsrmat}) is not known apriori. For analytic system $\Sigma$, $r$ can be set to the state dimension $n$, see Theorem $4.1$ in \cite{obsThesis}.
\end{remark}

\begin{remark}\label{algeom}
For a polynomial system $\Sigma$, observability has also been studied from the perspective of algebraic geometry, see \cite{gerbet2020global} and references therein.
\end{remark}

We adopt the use of local weak observability as the notion of nonlinear observability throughout the remainder of this paper.

\section{UNIFORM HYPERGRAPHS}\label{sec: uniform hypergraphs}
A \textit{undirected hypergraph} $\g=\{\Vs,\Es\}$ where $\Vs$ is a finite set and $\Es\subseteq \mathcal{P}(\Vs)\setminus \{\emptyset\}$, the power set of $\Vs$. The elements of $\Vs$ are called the nodes, and the elements of $\Es$ are called the hyperedges. A hypergraph is $k$-uniform if all hyperedges contain exactly $k$ vertices.
\subsection{Uniform Hypergraph Structure}

\begin{defn}
Let $\g= \{\Vs,\Es\}$ be a $k$-uniform hypergraph with $n=|\Vs|$ nodes. The \textit{adjacency tensor} $\tA\in\mathbb{R}^{n\times n\times \dots\times n}$ of $\g$ is a $k$-th order, $n$-dimensional, supersymmetric tensor is defined
\begin{equation}\label{eq:98}
    \tA_{j_1j_2\dots j_k} = 
    \begin{cases}
        \frac{1}{(k-1)!} & \text{if $\{j_1,j_2,\dots,j_k\}\in \Es$}\\
        0 &  \text{otherwise}\end{cases}.
\end{equation}

\end{defn}


We recall definitions of uniform hypergraph chain, ring, star and complete hypergraphs following \cite{chen2021controllability}.

\begin{defn}\label{def: hyperchain}
 A $k$-uniform \textit{hyperchain} is a sequence of $n$ nodes such that every $k$ consecutive nodes are adjacent, i.e., nodes $j,j+1,\dots,j+k-1$ are contained in one hyperedge for $j=1,2,\dots,n-k+1$.
\end{defn}

\begin{defn}
A $k$-uniform \textit{hyperring} is a sequence of $n$ nodes such that every $k$ consecutive nodes are adjacent, i.e., nodes $\sigma_n(j),\sigma_n(j+1),\dots,\sigma_n(j+k-1)$ are contained in one hyperedge for $j=1,2,\dots,n$, where $\sigma_n(j)=j$ for $j\leq n$ and $\sigma_n(j)=j-n$ for $j>n$.
\end{defn}

\begin{defn}
A $k$-uniform \textit{hyperstar} is a collection of $k-1$ internal nodes that are contained in all the hyperedges, and $n-k+1$ leaf nodes such that every leaf node is contained in one hyperedge with the internal nodes.
\end{defn}

\begin{defn}\label{def: complete hypergraph}
A $k-$uniform \textit{complete hypergraph} is a set of $n$ vertices with all $\binom{n}{k}$ possible hyperedges.
\end{defn}

See Fig. \ref{fig: toy HG} in Section \ref{sec:syntexample} for examples of these structures, and note that when $k=2$, definitions \ref{def: hyperchain} - \ref{def: complete hypergraph} are reduced to standard chains, rings, stars and complete graphs.

\subsection{Uniform Hypergraph Dynamics with Outputs}
We extent the homogeneous polynomial/multilinear time-invariant dynamics of a $k$-uniform hypergraph to include linear system outputs.


\begin{defn}
Given a $k$-uniform undirected hypergraph $\g$ with $n$ nodes, the dynamics of $\g$ with outputs $\yv\in \R^m$ is defined as
\begin{equation}\label{eq:hypdyn1}
    \Sigma\begin{cases}
        \dot{\xv} & = \fv(\xv)=\tA\xv^{k-1}\\
        \yv & = \gv(\xv)=\Cv\xv,        
    \end{cases}
\end{equation}
where $\tA\in\mathbb{R}^{n\times n\times \dots \times n}$ is the adjacency tensor of $\g$, and $\Cv\in\mathbb{R}^{m\times n}$ is the output matrix.
\end{defn}

See Fig. \ref{fig:1} for an example of uniform hypergraph and associated dynamics. All the interactions are characterized using multiplications instead of the additions that are typically used in a standard graph based representation. For detailed discussion on relationship between graph vs. hypergraph dynamic representation, see \cite{chen2021controllability}.

\section{OBSERVABILITY COMPUTATION FOR UNIFORM HYPERGRAPHS}\label{sec: obsv uniform hypergraphs}
In this section, we recast the homogeneous polynomial/multilinear hypergraph dynamical system in terms of the Kronecker product, derive a construction of the corresponding NOM, and propose a recursive algorithm to perform the construction.

The hypergraph dynamics (\ref{eq:hypdyn1}) can be expressed equivalently as the unfolded tensor $\tA$ contracted with the Kronecker exponentiation of the state vector:
\begin{equation}\label{eq:hypdynkron1}
    \Sigma\begin{cases}
        \dot{\xv} & = \fv(\xv) = \tA_{(p)}\xv^{[k-1]}\\
        \yv & = \gv(\xv) = \Cv\xv,
    \end{cases}
\end{equation}
where, $\tA_{(p)}\in \R^{n\times n^{k-1}}$ is the $p$-th mode unfolding of $\tA$. Since $\tA$ is a super-symmetric tensor, all $p$-th mode unfoldings give rise to the same matrix $\tA_{(p)}=\Av$.

Furthermore, $\fv$ and $\gv$ are homogeneous polynomials and hence analytic functions. Thus, one can in principle use Sedoglavic's algorithm \cite{sedoglavic2001probabilistic} for rank computation of NOM associated with above system, as performed in \cite{liu2013observability} and discussed in Section \ref{sec:obs}. Alternatively, algebraic geometric techniques for polynomial systems can also be used as indicated in the Remark \ref{algeom}. These approaches while general purpose tend to be computationally expensive. We develop a specialized framework exploiting structure of hypergraph dynamics introduced above for potentially more efficient observability computation.

\begin{figure}[t!]
\centering
\tcbox[colback=white,top=5pt,left=5pt,right=-5pt,bottom=5pt]{
\begin{tikzpicture}[scale=0.91, transform shape]
\node[vertex,text=white,scale=0.7] (v1) {1};
\node[vertex,above right = 11pt and 1.5pt of v1,text=white,scale=0.7] (v2) {2};
\node[vertex,below right = 11pt and 1.5pt of v2,text=white,scale=0.7] (v3) {3};

\node[vertex,below right = 3pt and 40pt of v2,white,scale=0.5] (v7){};
\node[vertex,below right = 3pt and 6pt of v2,white,scale=0.5] (v8){};
\node[vertex,below right = -9pt and 12pt of v2,white,scale=0.7,text=black] (v9){  $\dot{\textbf{x}}$=$\textbf{A}\textbf{x}$};
\node[rectangle,below right = -20pt and 40pt of v2,white,scale=0.7,text=black] (v6) {
$\Large
\begin{cases}
\dot{x}_1 = x_2+x_3\\
\dot{x}_2 = x_1 + x_3\\
\dot{x}_3 = x_1 + x_2
\end{cases}
$};
\node[vertex,above of=v2,node distance=20pt,white,text=black,scale=0.7] (text1) {\Large\textbf{A}};
\path [->,shorten >=1pt,shorten <=1pt, thick](v8) edge node[left] {} (v7);

\node[vertex,right of=v3,node distance=120pt,text=white,scale=0.7](w1) {1};
\node[vertex,above right = 11pt and 1.5pt of w1,text=white,scale=0.7] (w2) {2};
\node[vertex,below right = 11pt and 1.5pt of w2,text=white,scale=0.7] (w3) {3};
\node[vertex,below right = 3pt and 40pt of w2,white,scale=0.5] (w7){};
\node[vertex,below right = 3pt and 6pt of w2,white,scale=0.5] (w8){};
\node[vertex,below right = -10pt and 12pt of w2,white,scale=0.7,text=black] (w9){  $\dot{\textbf{x}}$=$\textsf{A}\textbf{x}^{2}$};

\node[rectangle,below right = -20pt and 40pt of w2,white,scale=0.7,text=black] (w6) {
$\Large
\begin{cases}
\dot{x}_1 = x_2x_3\\
\dot{x}_2 = x_1x_3\\
\dot{x}_3 = x_1x_2
\end{cases}
$};
\node[vertex,above of=w2,node distance=20pt,white,text=black,scale=0.7] (text1) {\Large\textbf{B}};

\path [->,shorten >=1pt,shorten <=1pt, thick](w8) edge node[left] {} (w7);

\begin{pgfonlayer}{background}
\draw[edge,color=orange,line width=8pt] (v1) -- (v2);
\draw[edge,color=red,line width=8pt] (v2) -- (v3);
\draw[edge,color=green,line width=8pt] (v3) -- (v1);

\draw[edge,color=orange,line width=8pt] (w1) -- (w2) -- (w3) -- (w1);
\end{pgfonlayer}

\end{tikzpicture}
}
\caption{Graphs versus uniform hypergraphs. (A) Standard graph with three nodes and edges $e_1=\{1,2\}$, $e_2=\{2,3\}$ and $e_3=\{1,3\}$, and its corresponding linear dynamics. (B) 3-uniform hypergraph with three nodes and a hyperedge $e_1=\{1,2,3\}$, and its corresponding nonlinear dynamics.}
\label{fig:1}
\end{figure}
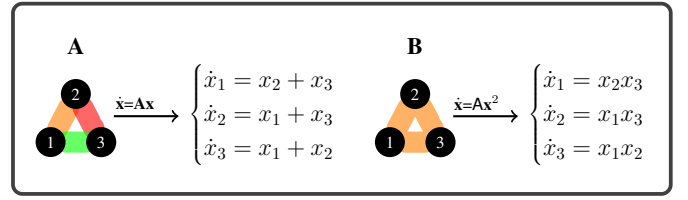

\subsection{Observability Criterion}
To determine the NOM (\ref{eq:obsrmat}) for the systems (\ref{eq:hypdyn1}) and (\ref{eq:hypdynkron1}), we compute 
the Lie derivatives of the system output along the flow of the system state:
\begin{eqnarray*}
    L_{\fv}^0\gv(\xv) & =& \Cv\xv,\\
    L_{\fv}^1\gv(\xv) & = & \frac{d}{dt}\Cv\xv=\Cv\Av\xv^{[k-1]},\\
    L_{\fv}^2\gv(\xv) & = & \frac{d}{dt}\Cv\Av\xv^{[k-1]}\\
              & = & \Cv\Av\frac{d}{dt}\bigg(\overbrace{\xv\otimes \xv\dots\otimes\xv}^{\text{$k-1$ times}}\bigg)\\
              & = & \Cv\Av\bigg(\dot{\xv}\otimes\dots\otimes \xv+\cdots+\xv\otimes\dots\otimes\dot{\xv}\bigg)\\ 
              & = & \Cv\Av\bigg(\sum \xv\otimes\dots\otimes \Av\xv^{[k-1]}\otimes\dots\otimes \xv\bigg)\\
              & = & \Cv\Av\bigg[\bigg(\sum \Id\otimes\dots\otimes \Av\otimes\dots\otimes \Id\bigg) \xv^{2k-3}\bigg]\\
              & = & \Cv\Av\Bv_2 \xv^{2k-3},\\
    & \vdots &\\
    L_{\fv}^n\gv(\xv) & = & \Cv\Av\Bv_2\dots \Bv_n \xv^{[nk-(2n-1)]} \quad \forall n>2,
\end{eqnarray*}
where $\Bv_p$ is given by,
\begin{equation}\label{eq:B}
\Bv_p=\sum_{i=1}^{(p-1)k-(2p-3)} \overbrace{\Id\otimes\dots\otimes \underbrace{\Av}_{i\text{-th pos.}}\otimes\dots\otimes \Id}^{(p-1)k-(2p-3) \text{times}}.
\end{equation}
The NOM may then be written as
\begin{equation}\label{eq:hgobsv}
    \mathcal{O}(\xv)=\nabla_{\xv}\begin{pmatrix}
    \Cv\xv\\
    \Cv\Av\xv^{[k-1]}\\
    \Cv\Av\Bv_2\xv^{[2k-3]}\\
    \vdots\\
    \Cv\Av\Bv_2\dots \Bv_n \xv^{[nk-(2n-1)]}
    \end{pmatrix},
\end{equation}
where, we have used $r=n$ as per Remark \ref{remrn}. From Theorem \ref{thm1}, when $rank(\mathcal{O}(\xv))=\dim(\xv),$ systems (\ref{eq:hypdyn1}) and (\ref{eq:hypdynkron1}) are observable.

\begin{remark}
For the case $k=2$, the hypergraph reduces to a graph with adjacency matrix $\textsf{A}$ and linear dynamics. Then, $\mathcal{O}(\xv)$ reduces to the Kalman observability matrix, and our observability test is equivalent to the famous Kalman-rank condition.
\end{remark}

\subsection{Computational Framework}
The NOM (\ref{eq:hgobsv}) can be expressed in the form, 
\begin{equation*}\label{eq:hgobsv1}
    \mathcal{O}(\xv)=\nabla_{\xv}\begin{pmatrix}
    \Cv\Jv_0(\xv)\\
    \Cv\Jv_1(\xv)\\
    \Cv\Jv_2(\xv)\\
    \vdots\\
    \Cv\Jv_n(\xv)
    \end{pmatrix},
\end{equation*}
where, $\Jv_i(\xv)\in \R^n$ are vectors defined as,
\begin{eqnarray}
\Jv_0(\xv)&=&\xv,\notag\\
\Jv_1(\xv)&=&\Av\xv^{[k-1]},\notag
\end{eqnarray}
and
\begin{eqnarray*}
\Jv_p(\xv)&=&\Av\Bv_2\dots \Bv_p \xv^{[pk-(2p-1)]},\label{eq:J}
\end{eqnarray*}
for $p=2,\cdots,n-1$. It is computationally infeasible to construct the $\Bv_p$ matrices explicitly, so we apply the mixed product property to evaluate $\Bv_p\xv^{[pk-(2p-1)]}$ as,
\begin{eqnarray}\label{eq:mixedProduct1}
&&\Bv_p\xv^{[pk-(2p-1)]}\notag\\
&=&\sum_{i=1}^{(p-1)k-(2p-3)} \overbrace{\xv\otimes\dots\otimes \underbrace{\Av\xv^{[k-1]}}_{i\text{-th vpos.}}\otimes\dots\otimes \xv}^{(p-1)k-(2p-3) \text{times}}.
\end{eqnarray}
The mixed product property can be recursively exploited  to compute products $\Bv_2\cdots \Bv_p\xv^{[pk-(2p-1)]}$ appearing in $\Jv_p(\xv).$

From Eqn.~\ref{eq:mixedProduct1}, define sets $S_1,\dots, S_{(p-1)k-(2p-3)}$ such that
\begin{equation*}
    S_i=\bigg\{{\xv,\dots, \underbrace{\Av \xv^{[k-1]}}_{i\text{-th pos.}},\dots, \xv}^{}\bigg\},
\end{equation*}
where $|S_i|=(p-1)k-(2p-3).$ Given all sets $S_i,$ the calculation of $\Bv_{p-1}\Bv_p\xv^{[pk-(2p-1)]}$ follows a similar procedure to Eqn.~\ref{eq:mixedProduct1} to obtain the result that $\Bv_{p-1}\Bv_p\xv^{[pk-(2p-1)]}$ may be written as
\begin{multline}\label{eq: recursive mixed product}
    \sum_{i=1}^{\substack{(p-2)k\\-(2p-5)}}\ \sum_{j=1}^{\substack{(p-1)k\\-(2p-3)}}\bigg(S_{j,1}\otimes\dots
    \otimes\underbrace{\Av\big(S_{j,i}
    \otimes\dots\otimes S_{j,i+k-2}\big)}_{i\text{-th pos.}}\\
    \otimes\dots\otimes S_{j,(p-1)k-(2p-3)}\bigg),
\end{multline}
where there are $(p-2)k-(2p-5)$ vectors in $\R^n$ that are Kronecker product within each calculation of the inner sum. These $(p-2)k-(2p-5)$ vectors form elements in the recursively calculated sets $S_i$. Algorithm \ref{alg:Jp} performs the recursive calculation of $\Jv_p(\xv).$

\begin{thm}
Recursive application of the mixed product property of Kronecker products computes $
\Jv_p(\xv)$ without computing a Kronecker exponentiation larger than $\xv^{[k-1]}.$
\end{thm}

\begin{proof}
Recursive use of the mixed product property in computing $\Jv_p(\xv)$ generates a series of sets $S_i'$ for $i=1,\dots,(p-1)k-(2p-3)$ where every term in the sets is a vector in $\R^n.$ From Eqn.~(\ref{eq: recursive mixed product}), each term in the sets $S_i'$ is taken directly from $S_i$ or through the multiplication of $\Av$ with the Kronecker product of $k-1$ vectors in $S_i$ (Step 9 in Algorithm 1), which requires computing vectors of size at most $n^{k-1}$ with the Kronecker product.
\end{proof}

Addressing the locality of the system state remains a central challenge in determining nonlinear observability. To provide a notion of global observability, Algorithm \ref{alg:Jp} is applied to symbolic state vectors such that the rank condition of the symbolic NOM is a test of observability for all states of the system. 
Since the complexity of symbolic operations is not fixed, Algorithm \ref{alg:Jp} is optimized to minimized Kronecker exponentiation rather than to reduce the number of floating point operations (FLOPs). 
If we were to test for local observability at a numeric state, we would utilize a Lyapanov-like matrix solver to evaluate the NOM instead.

\color{black}

\begin{algorithm}[t]
\caption{{RecursiveJp($\Av,\ p,\ k, S_j)$}}
\label{alg:Jp}
\begin{algorithmic}[1]
\IF{$p=1$}
    \STATE $\Jv_p(\xv)=\Av\big(S_{j,1}\otimes\dots\otimes S_{j,(k-1)}\big)$
    \STATE \textbf{return:} $\Jv_p(\xv)$
\ENDIF
\STATE $b=(p-1)k-(2p-3)$
\STATE $\Jv_p(\xv)=0$
\FOR{$i=1,\dots,b$}
    \STATE $S_i'=\{S_{j,1},\dots,\overbrace{\Av\big(S_{j,i}\otimes\dots\otimes S_{j,i+k-2}\big)}^{i\text{th pos.}},$ $\dots, S_{j,pk-(2p-1)}\}$
    \STATE $\Jv_p(\xv) =\Jv_p(\xv)\ +$ {RecursiveJp($\Av,\ p-1, k, S_i'$)}
\ENDFOR
\STATE \textbf{return:} $\Jv_p(\xv)$
\end{algorithmic}
\end{algorithm}

\section{MINIMUM OBSERVABLE NODE SELECTION}\label{sec: MONS}
Finding the minimum set of observable nodes (MON) is a combinatorial optimization problem, and is in general intractable using brute-force search. We provide a greedy heuristic approach for estimating the MON of a uniform hypergraph in which nodes are chosen as measurements based on the maximum change in the rank of $\mathcal{O}_D(\xv)$, see Algorithm \ref{alg:2}. Here $\mathcal{O}_D(\xv)$ denotes the NOM for nodes in the index set $D=\{i_1,\cdots,i_m\}\subset S=\{1,\cdots,n\}$ as the outputs, and is given by,
\begin{equation*}\label{eq:OD}
\mathcal{O}_D(\xv)=\left(
         \begin{array}{c}
          \mathcal{O}_{i_1}(\xv) \\
          \mathcal{O}_{i_2}(\xv)  \\
          \vdots \\
          \mathcal{O}_{i_m}(\xv)  \\
         \end{array}
       \right),
\end{equation*}
where, construction of $\mathcal{O}_i(\xv),i\in D$ is as follows. Let the gradient of entries $\Jv_{p,i}(x),i=1,\cdots,n$ of vector $\Jv_{p},$ for $p=1,\cdots,n$ be denoted,
\begin{equation*}\label{eq:Jo}
\overline{\Jv}_{p,i}(\xv)=\Cv_i\Jv_{p,i}(\xv).
\end{equation*}
where the output matrix $\Cv_i\in\Z^{1\times n}$ observes only the $i$-th node i.e. $1$ at $i$-th entry and zero otherwise.
Then
\begin{equation*}\label{eq:O}
\mathcal{O}_i(\xv)=\nabla_{\xv} \left(
         \begin{array}{c}
          \Cv_i\Jv_0(x) \\
          \Cv_i\Jv_2(x) \\
          \vdots \\
          \Cv_i\Jv_{n-1}(x) \\
         \end{array}
       \right)= \left(
         \begin{array}{c}
          \overline{\Jv}_{0,i}(x) \\
          \overline{\Jv}_{1,i}(x) \\
          \vdots \\
          \overline{\Jv}_{n-1,i}(x) \\
         \end{array}
       \right),
\end{equation*}
is the NOM  with $i$-th node as the measurement.

Algorithm \ref{alg:2} provides a greedy approach to selecting the MON of a hypergraph. The individual observation matrices are provided as input and may be computed according to Algorithm \ref{alg:Jp}. The conditional statement in line 2 implements Theorem \ref{thm1} to determine the size at which the system is observable with only the vertices contained in $D.$ Line 4 computes the greedy heuristic to maximize the rank of the NOM, and the rank calculations can be performed either numerically or symbolically. The choice to execute Algorithm \ref{alg:2} numerically versus symbolically is the dominant factor in determining the time complexity of Algorithm \ref{alg:2}. When the cost of checking the rank in line 4 is $r,$ the run time of Algorithm \ref{alg:2} is $\mathcal{O}(rn^{2}).$
\color{black}



\begin{algorithm}[t]
\caption{GreedyMON($\mathcal{O}_i(\xv)$ for all $i=1,\cdots,n$)}
\label{alg:2}
\begin{algorithmic}[1]
\STATE{Let $S=\{1,2,\dots,n\}$ and $D=\emptyset$}\\
\WHILE{$\mbox{rank}(\mathcal{O}_D(\xv))<n$}
    \FOR{$s\in S\setminus D$}
    \STATE{Compute $\Delta(s)=\text{rank}(\mathcal{O}_{D\cup \{s\}}(\xv))-\text{rank}(\mathcal{O}_D(\xv))$}\\
    \ENDFOR
    \STATE{Set $s^{*} = \text{argmax}_{s\in S\setminus D}\Delta(s)$}\\
    \STATE{Set $D=D\cup \{s^*\}$}\\
\ENDWHILE
\RETURN The set $D$.
\end{algorithmic}
\end{algorithm}

\section{NUMERICAL RESULTS}\label{sec: numerical results}

We demonstrate the identification of MON on uniform hypergraph chain, rings, and stars as well as a  hypergraph constructed from time series data. These calculations were performed symbolically with MATLAB R2022b.

\subsection{Synthetic Uniform Hypergraphs}\label{sec:syntexample}
We identified the MON set for uniform hypergraph chains, rings and stars and $k=2,\dots,n$ with $n=3,\dots,7$. For the hyperstar, the size of the MON increases with $n$ and decreases with $k$. As examples, in Fig. \ref{fig: toy HG}, six hypergraphs are shown with the identified MON. Future work aims to develop a theoretical characterization of the MON for these types of hypergraphs. 

\input{figMON}

\subsection{Mouse Neuron Endomicroscopy Hypergraph}

Hypothalamus neural activity during a feeding, fasting, and refeeding experiment was monitored with endomicroscopy to generate a time series data set \cite{sweeney2021network}. Similar to \cite{chen2021controllability} and \cite{surana2022hypergraph}, we construct 3 hypergraph representations of the activity of $15$ neurons during the different phases of the experiment. First, we compute the multi-correlation of all pairs of 3 neurons, which is defined 
\begin{equation}
    \rho = (1-\det(\mathbf{R}))^{1/2},
\end{equation}
where $\mathbf{R}\in\R^{3\times 3}$ is the correlation matrix among 3 neurons \cite{wang2014measures}. When the multi-correlation $\rho$ is greater than a prescribed threshold, we define a hyperedge among the 3 vertices. Following \cite{chen2021controllability}, we used a threshold of $0.95$. 

For each of the three hypergraphs, we identified the MON. Fig. \ref{fig:mouse} depicts the hypergraph structure during each phase of the feeding experiment and depicts the MON. A similar correlation, thresholding, and graph construction was performed on all 3 phases of the experiment to identify the linear MON. Across all phases of the experiment, the MON size is reduced for hypergraphs as opposed to graphs. During the fast phase of the experiment, the multi-correlation among all neurons decreases, which results in a less connected hypergraph and an increased size of the MON. Given that the number of observed nodes on the connected component is minimal, it appears that the size of the MON set is largely driven by hypergraph connectivity. While the hypergraph MON sets is greatly reduced compared to the graph MON sets during all three phases of the experiment, the size of the MON set is the same order of magnitude as the minimum control node sets identified on this data in \cite{chen2021controllability}.

\input{figMouse}



\section{CONCLUSION}\label{sec: conclusion}
In this paper, we proposed a framework to study observability for uniform hypergraphs. We defined a canonical multilinear dynamical system with linear outputs using uniform hypergraph adjacency tensor leading to a homogeneous polynomial system. We derived the NOM for assessing the local weak observability of this resulting system.  We also proposed a recursive technique for efficient computation of the NOM, and a greedy heuristic to determine the MON. We demonstrated our approach numerically on different hypergraph topologies, and hypergraphs derived from an experimental mouse endomicroscopy dataset.

In the future, we plan to perform theoretical analysis for determining MON for different hypergraph topologies and exploring the role of symmetry, and to extend the proposed framework for non-uniform and directed hypergraphs. We also hope to further improve efficiency of the observability computations to scale to large hypergraphs which often arise in practise.


\section*{ACKNOWLEDGMENTS}
This material is based upon work supported by the Air Force Office of Scientific Research under award numbers FA9550-22-1-0215 and FA9550-23-1-0400,
by NSF grant DMS-2103026, a MathWorks Fellowship to the Rajapakse Lab (IR), and National Institute of General Medical Sciences under award number GM150581 (JP). Any opinions, finding, and conclusions or recommendations expressed in this material are those of the author(s) and do not necessarily reflect the views of the United States Air Force.

\bibliographystyle{IEEEtran}
\bibliography{hypergraph}

\end{document}

%% file: figMON.tex
\input{colors}
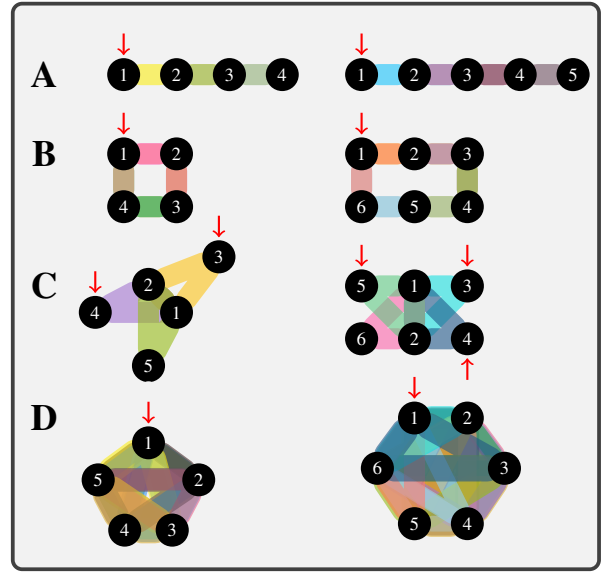
\begin{figure}[t!]
\centering
\tcbox[colback=mygray,top=0pt,left=-2pt,right=0pt,bottom=5pt]{
\begin{tikzpicture}[scale=1, transform shape]
\node[vertex,text=white,scale=0.7] (v1) {1};
\node[vertex,above of=v1,node distance=20pt,mygray,scale=0.7] (v1a) {};
\path [->,red, shorten >=1pt,shorten <=1pt, thick](v1a) edge node[left] {} (v1);
\node[vertex,right of=v1,text=white,scale=0.7,node distance = 20pt] (v2) {2};
\node[vertex,right of=v2,text=white,scale=0.7,node distance = 20pt] (v3) {3};
\node[vertex,right of=v3,text=white,scale=0.7,node distance = 20pt] (v4) {4};
\node[vertex,left of=v1,node distance=30pt,mygray,text=black,scale=0.7] (text1) {\huge\textbf{A}};

\node[vertex,right of=v4,text=white,scale=0.7,node distance=30pt] (v5) {1};
\node[vertex,above of=v5,node distance=20pt,mygray,scale=0.7] (v5a) {};
\path [->,red, shorten >=1pt,shorten <=1pt, thick](v5a) edge node[left] {} (v5);
\node[vertex,right of=v5,text=white,scale=0.7,node distance = 20pt] (v6) {2};
\node[vertex,right of=v6,text=white,scale=0.7,node distance = 20pt] (v7) {3};
\node[vertex,right of=v7,text=white,scale=0.7,node distance = 20pt] (v8) {4};
\node[vertex,right of=v8,text=white,scale=0.7,node distance = 20pt] (v9) {5};

\node[vertex,below of=v1,text=white,scale=0.7,node distance=30pt] (a1) {1};
\node[vertex,above of=a1,node distance=20pt,mygray,scale=0.7] (a1a) {};
\path [->,red, shorten >=1pt,shorten <=1pt, thick](a1a) edge node[left] {} (a1);
\node[vertex,right of=a1,text=white,scale=0.7,node distance = 20pt] (a2) {2};
\node[vertex,below of=a2,text=white,scale=0.7,node distance = 20pt] (a3) {3};
\node[vertex,left of=a3,text=white,scale=0.7,node distance = 20pt] (a4) {4};

\node[vertex,below of=v5,text=white,scale=0.7,node distance=30pt] (a5) {1};
\node[vertex,above of=a5,node distance=20pt,mygray,scale=0.7] (a5a) {};
\path [->,red, shorten >=1pt,shorten <=1pt, thick](a5a) edge node[left] {} (a5);
\node[vertex,right of=a5,text=white,scale=0.7,node distance = 20pt] (a6) {2};
\node[vertex,right of=a6,text=white,scale=0.7,node distance = 20pt] (a7) {3};
\node[vertex,below of=a7,text=white,scale=0.7,node distance = 20pt] (a8) {4};
\node[vertex,left of=a8,text=white,scale=0.7,node distance = 20pt] (a9) {5};
\node[vertex,left of=a9,text=white,scale=0.7,node distance = 20pt] (a10) {6};

\node[vertex,left of=a1,node distance=30pt,mygray,text=black,scale=0.7] (text2.1) {\huge\textbf{B}};

\node[vertex,below of=a3,text=white,scale=0.7,node distance=40pt] (c1) {1};
\node[vertex,above left of=c1,text=white,scale=0.7,node distance=15pt] (c2) {2};
\node[vertex,above right=12 pt and 7 pt of c1,text=white,scale=0.7,node distance=20pt] (c3) {3};
\node[vertex,above of=c3,node distance=20pt,mygray,scale=0.7] (c3a) {};
\path [->,red, shorten >=1pt,shorten <=1pt, thick](c3a) edge node[left] {} (c3);
\node[vertex,left= 18pt of c1,text=white,scale=0.7] (c4) {4};
\node[vertex,above of=c4,node distance=20pt,mygray,scale=0.7] (c4a) {};
\path [->,red, shorten >=1pt,shorten <=1pt, thick](c4a) edge node[left] {} (c4);
\node[vertex,below=18pt of c2,text=white,scale=0.7] (c5) {5};
\node[vertex,below of=text2.1,node distance=50pt,mygray,text=black,scale=0.7] (text2.2) {\huge\textbf{C}};

\node[vertex,below of=a9,text=white,scale=0.7,node distance=30pt] (c6) {1};
\node[vertex,right of=c6,text=white,scale=0.7,node distance=20pt] (c7) {3};
\node[vertex,above of=c7,node distance=20pt,mygray,scale=0.7] (c7a) {};
\path [->,red, shorten >=1pt,shorten <=1pt, thick](c7a) edge node[left] {} (c7);
\node[vertex,below of=c7,text=white,scale=0.7,node distance=20pt] (c8) {4};
\node[vertex,below of=c8,node distance=20pt,mygray,scale=0.7] (c8a) {};
\path [->,red, shorten >=1pt,shorten <=1pt, thick](c8a) edge node[left] {} (c8);
\node[vertex,left of=c8,text=white,scale=0.7,node distance=20pt] (c9) {2};
\node[vertex,left of=c9,text=white,scale=0.7,node distance=20pt] (c10) {6};
\node[vertex,above of=c10,text=white,scale=0.7,node distance=20pt] (c11) {5};
\node[vertex,above of=c11,node distance=20pt,mygray,scale=0.7] (c11a) {};
\path [->,red,shorten >=1pt,shorten <=1pt, thick](c11a) edge node[left] {} (c11);

\node[vertex,below of=c2,text=white,scale=0.7,node distance=60pt] (d1) {1};
\node[vertex,below right= 5pt and 10pt of d1,text=white,scale=0.7,node distance=20pt] (d2) {2};
\node[vertex,below left= 10pt and 1pt of d2,text=white,scale=0.7,node distance=20pt] (d3) {3};
\node[vertex,below left= 5pt and 10pt of d1,text=white,scale=0.7,node distance=20pt] (d5) {5};
\node[vertex,below right= 10pt and 1pt of d5,text=white,scale=0.7,node distance=20pt] (d4) {4};
\node[vertex,above of=d1,node distance=20pt,mygray,scale=0.7] (d1a) {};
\path [->,red,shorten >=1pt,shorten <=1pt, thick](d1a) edge node[left] {} (d1);
\node[vertex,below of=text2.2,node distance=50pt,mygray,text=black,scale=0.7] (text2.3) {\huge\textbf{D}};

\node[vertex,below of=c9,text=white,scale=0.7,node distance=30pt] (e1) {1};
\node[vertex,right of=e1,text=white,scale=0.7,node distance=20pt] (e2) {2};
\node[vertex,below right= 10pt and 5pt of e2,text=white,scale=0.7,node distance=20pt] (e3) {3};
\node[vertex,below of=e2,text=white,scale=0.7,node distance=40pt] (e4) {4};
\node[vertex,below of=e1,text=white,scale=0.7,node distance=40pt] (e5) {5};
\node[vertex,below left= 10pt and 5pt of e1,text=white,scale=0.7,node distance=20pt] (e6) {6};
\node[vertex,above of=e1,node distance=20pt,mygray,scale=0.7] (e1a) {};
\path [->, red, shorten >=1pt,shorten <=1pt, thick](e1a) edge node[left] {} (e1);

\begin{pgfonlayer}{background}
\begin{scope}[transparency group,opacity=.9]
\draw[edge,color=aureolin, line width=8pt] (v1) -- (v2) -- (v3);
\draw[edge,color=asparagus, line width=8pt] (v2) -- (v3) -- (v4);

\draw[edge,color=capri, line width=8pt] (v5) -- (v6) -- (v7);
\draw[edge,color=darkpink, line width=8pt] (v6) -- (v7) -- (v8);
\draw[edge,color=eggplant, line width=8pt] (v7) -- (v8) -- (v9);

\draw[edge,color=fuchsia, line width=8pt] (a1) -- (a2) -- (a3);
\draw[edge,color=goldenrod, line width=8pt] (a2) -- (a3) -- (a4);
\draw[edge,color=green(pigment), line width=8pt] (a3) -- (a4) -- (a1);
\draw[edge,color=salmon, line width=8pt] (a4) -- (a1) -- (a2);

\draw[edge,color=amber, line width=8pt] (a5) -- (a6) -- (a7);
\draw[edge,color=amethyst, line width=8pt] (a6) -- (a7) -- (a8);
\draw[edge,color=applegreen, line width=8pt] (a7) -- (a8) -- (a9);
\draw[edge,color=ashgrey, line width=8pt] (a8) -- (a9) -- (a10);
\draw[edge,color=babyblue, line width=8pt] (a9) -- (a10) -- (a5);
\draw[edge,color=bittersweet, line width=8pt] (a10) -- (a5) -- (a6);

\draw[edge,color=amber, line width=8pt] (c1) -- (c2) -- (c3) -- (c1);
\draw[edge,color=amethyst, line width=8pt] (c1) -- (c2) -- (c4) -- (c1);
\draw[edge,color=applegreen, line width=8pt] (c1) -- (c2) -- (c5) -- (c1);

\draw[edge,color=blue-green, line width=8pt] (c6) -- (c9) -- (c7) -- (c6);
\draw[edge,color=darkcerulean, line width=8pt] (c6) -- (c9) -- (c8) -- (c6);
\draw[edge,color=brilliantrose, line width=8pt] (c6) -- (c9) -- (c10) -- (c6);
\draw[edge,color=emerald, line width=8pt] (c6) -- (c9) -- (c11) -- (c6);

\draw[edge,color=sandstorm, line width=6pt] (d1) -- (d2) -- (d3) -- (d1);
\draw[edge,color=teal, line width=7pt] (d1) -- (d2) -- (d4) -- (d1);
\draw[edge,color=sealbrown, line width=8pt] (d1) -- (d2) -- (d5) -- (d1);
\draw[edge,color=bleudefrance, line width=9pt] (d1) -- (d3) -- (d4) -- (d1);
\draw[edge,color=aureolin, line width=10pt] (d1) -- (d3) -- (d5) -- (d1);
\draw[edge,color=asparagus, line width=6pt] (d1) -- (d4) -- (d5) -- (d1);
\draw[edge,color=capri, line width=7pt] (d2) -- (d3) -- (d4) -- (d2);
\draw[edge,color=darkpink, line width=8pt] (d2) -- (d3) -- (d5) -- (d2);
\draw[edge,color=eggplant, line width=9pt] (d2) -- (d4) -- (d5) -- (d2);
\draw[edge,color=goldenrod, line width=10pt] (d3) -- (d4) -- (d5) -- (d3);

\draw[edge,color=sandstorm, line width=6pt] (e1) -- (e2) -- (e3) -- (e1);
\draw[edge,color=teal, line width=7pt] (e1) -- (e2) -- (e4) -- (e1);
\draw[edge,color=sealbrown, line width=8pt] (e1) -- (e2) -- (e5) -- (e1);
\draw[edge,color=bleudefrance, line width=9pt] (e1) -- (e3) -- (e4) -- (e1);
\draw[edge,color=aureolin, line width=10pt] (e1) -- (e3) -- (e5) -- (e1);
\draw[edge,color=asparagus, line width=6pt] (e1) -- (e4) -- (e5) -- (e1);
\draw[edge,color=capri, line width=7pt] (e2) -- (e3) -- (e4) -- (e2);
\draw[edge,color=darkpink, line width=8pt] (e2) -- (e3) -- (e5) -- (e2);
\draw[edge,color=eggplant, line width=9pt] (e2) -- (e4) -- (e5) -- (e2);
\draw[edge,color=goldenrod, line width=10pt] (e3) -- (e4) -- (e5) -- (e3);
\draw[edge,color=green(pigment), line width=6pt] (e6) -- (e2) -- (e3) -- (e6);
\draw[edge,color=salmon, line width=7pt] (e6) -- (e2) -- (e4) -- (e6);
\draw[edge,color=amber, line width=8pt] (e6) -- (e2) -- (e5) -- (e6);
\draw[edge,color=amethyst, line width=9pt] (e6) -- (e3) -- (e4) -- (e6);
\draw[edge,color=applegreen, line width=10pt] (e6) -- (e3) -- (e5) -- (e6);
\draw[edge,color=ashgrey, line width=6pt] (e6) -- (e4) -- (e5) -- (e6);
\draw[edge,color=babyblue, line width=7pt] (e6) -- (e1) -- (e4) -- (e6);
\draw[edge,color=bittersweet, line width=8pt] (e6) -- (e1) -- (e5) -- (e6);
\draw[edge,color=blue-green, line width=9pt] (e6) -- (e1) -- (e2) -- (e6);
\draw[edge,color=darkcerulean, line width=10pt] (e6) -- (e1) -- (e3) -- (e6);

\end{scope}
\end{pgfonlayer}

\end{tikzpicture}
}
\caption{MON of 3-uniform hyperchains, hyperrings and hyperstars. The nodes with arrows are denoted as the MON nodes. (A) $3-$uniform hyperchain on $n=4$ (left) and $n=5$ (right) vertices. (B) $3-$uniform hyperring on $n=4$ (left) and $n=6$ (right) vertices. (C) $3-$uniform hyperstar on $n=5$ (left) and $n=6$ (right) vertices. (D) $3-$uniform complete hypergraphs with $n=5$ (left) and $n=6$ (right) vertices. The red arrows indicate the MON.}
\label{fig: toy HG}
\end{figure}

%% file: figMouse.tex
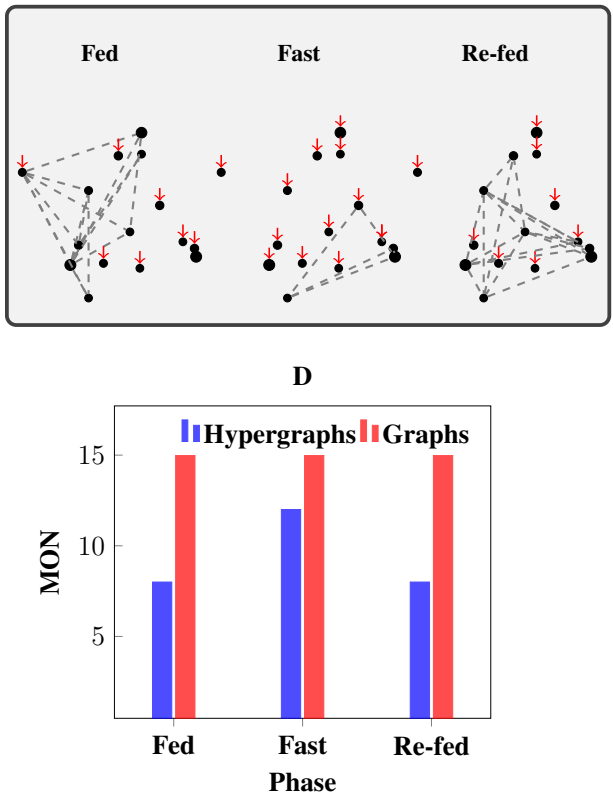
\begin{figure}[t]
    \centering
    \tcbox[colback=mygray,top=0pt,left=-1pt,right=1pt,bottom=5pt]{
    \begin{tikzpicture}[scale=0.022]
    \node[circle,inner sep=0.8] at  (60,150) {\small\textbf{Fed}};
    
    \node[circle,draw,fill=black,color=black,inner sep=4.12311/4] (v1) at  (110,36) {};
    \node[circle,draw,fill=black,color=black,inner sep=4.24264/4] (v2) at  (13,78) {};
    \node[circle,draw,fill=black,color=black,inner sep=4.47214/4] (v3) at  (96,58) {};
    \node[circle,draw,fill=black,color=black,inner sep=5.65685/4] (v4) at  (85,102) {};
    \node[circle,draw,fill=black,color=black,inner sep=5.83095/4] (v5) at  (42,22) {};
    \node[circle,draw,fill=black,color=black,inner sep=4.47214/4] (v6) at  (47,34) {};
    \node[circle,draw,fill=black,color=black,inner sep=4.24264/4] (v7) at  (53,67) {};
    \node[circle,draw,fill=black,color=black,inner sep=4.12311/4] (v8) at  (85,89) {};
    \node[circle,draw,fill=black,color=black,inner sep=4.12311/4] (v9) at  (78,42) {};
    \node[circle,draw,fill=black,color=black,inner sep=4.12311/4] (v10) at  (53,2) {};
    \node[circle,draw,fill=black,color=black,inner sep=4.24264/4] (v11) at  (62,23) {};
    \node[circle,draw,fill=black,color=black,inner sep=4.47214/4] (v14) at  (71,88) {};
    \node[circle,draw,fill=black,color=black,inner sep=4.47214/4] (v16) at  (117,32) {};
    \node[circle,draw,fill=black,color=black,inner sep=5.83095/4] (v19) at  (118,27) {};
    \node[circle,draw,fill=black,color=black,inner sep=4.12311/4] (v20) at  (84,20) {};
    
    \draw [dashed, thick,gray] (v2) -- (v4);
    \draw [dashed, thick,gray] (v2) -- (v7);
    \draw [dashed, thick,gray] (v2) -- (v5);
    \draw [dashed, thick,gray] (v2) -- (v6);
    \draw [dashed, thick,gray] (v2) -- (v9);
    \draw [dashed, thick,gray] (v2) -- (v10);
    \draw [dashed, thick,gray] (v4) -- (v5);
    \draw [dashed, thick,gray] (v4) -- (v9);
    \draw [dashed, thick,gray] (v5) -- (v6);
    \draw [dashed, thick,gray] (v5) -- (v7);
    \draw [dashed, thick,gray] (v5) -- (v8);
    \draw [dashed, thick,gray] (v5) -- (v9);
    \draw [dashed, thick,gray] (v5) -- (v10);
    \draw [dashed, thick,gray] (v6) -- (v7);
    \draw [dashed, thick,gray] (v6) -- (v8);
    \draw [dashed, thick,gray] (v7) -- (v10);
    
    \node[vertex,above of=v1,node distance=10pt,scale=0.7, white,opacity=0] (w1) {};
    \node[vertex,above of=v2,node distance=10pt,scale=0.7, white,opacity=0] (w2) {};
    \node[vertex,above of=v3,node distance=10pt,scale=0.7, white,opacity=0] (w3) {};
    \node[vertex,above of=v11,node distance=10pt,scale=0.7, white,opacity=0] (w11) {};
    \node[vertex,above of=v14,node distance=10pt,scale=0.7, white,opacity=0] (w14) {};
    \node[vertex,above of=v16,node distance=10pt,scale=0.7, white,opacity=0] (w16) {};
    \node[vertex,above of=v20,node distance=10pt,scale=0.7, white,opacity=0] (w20) {};
    \node[vertex,above of=v4,node distance=10pt,scale=0.7, white,opacity=0] (w4) {};
    \node[vertex,above of=v8,node distance=10pt,scale=0.7, white,opacity=0] (w8) {};
    
    \draw [->, red, line width=0.6pt] (w2) -- (v2);
    \draw [->, red, line width=0.6pt] (w1) -- (v1);
    \draw [->, red, line width=0.6pt] (w3) -- (v3);
    \draw [->, red, line width=0.6pt] (w11) -- (v11);
    \draw [->, red, line width=0.6pt] (w14) -- (v14);
    \draw [->, red, line width=0.6pt] (w16) -- (v16);
    \draw [->, red, line width=0.6pt] (w20) -- (v20);

     \end{tikzpicture}
     \begin{tikzpicture}[scale=0.022]
     \node[circle,inner sep=0.8] at  (60,150) {\small\textbf{Fast}};
    \node[circle,draw,fill=black,color=black,inner sep=4.12311/4] (v1) at  (110,36) {};
    \node[circle,draw,fill=black,color=black,inner sep=4.24264/4] (v2) at  (13,78) {};
    \node[circle,draw,fill=black,color=black,inner sep=4.47214/4] (v3) at  (96,58) {};
    \node[circle,draw,fill=black,color=black,inner sep=5.65685/4] (v4) at  (85,102) {};
    \node[circle,draw,fill=black,color=black,inner sep=5.83095/4] (v5) at  (42,22) {};
    \node[circle,draw,fill=black,color=black,inner sep=4.47214/4] (v6) at  (47,34) {};
    \node[circle,draw,fill=black,color=black,inner sep=4.24264/4] (v7) at  (53,67) {};
    \node[circle,draw,fill=black,color=black,inner sep=4.12311/4] (v8) at  (85,89) {};
    \node[circle,draw,fill=black,color=black,inner sep=4.12311/4] (v9) at  (78,42) {};
    \node[circle,draw,fill=black,color=black,inner sep=4.12311/4] (v10) at  (53,2) {};
    \node[circle,draw,fill=black,color=black,inner sep=4.24264/4] (v11) at  (62,23) {};
    \node[circle,draw,fill=black,color=black,inner sep=4.47214/4] (v14) at  (71,88) {};
    \node[circle,draw,fill=black,color=black,inner sep=4.47214/4] (v16) at  (117,32) {};
    \node[circle,draw,fill=black,color=black,inner sep=5.83095/4] (v19) at  (118,27) {};
    \node[circle,draw,fill=black,color=black,inner sep=4.12311/4] (v20) at  (84,20) {};

    \draw [dashed, thick,gray] (v3) -- (v10);
    \draw [dashed, thick,gray] (v3) -- (v19);
    \draw [dashed, thick,gray] (v10) -- (v16);
    \draw [dashed, thick,gray] (v10) -- (v19);
    \draw [dashed, thick,gray] (v16) -- (v19);

    \node[vertex,above of=v3,node distance=10pt,scale=0.7, white,opacity=0] (w3) {};
    \draw [->, red, line width=0.6pt] (w3) -- (v3);

    \node[vertex,above of=v1,node distance=10pt,scale=0.7, white,opacity=0] (w1) {};
    \node[vertex,above of=v2,node distance=10pt,scale=0.7, white,opacity=0] (w2) {};
    \node[vertex,above of=v4,node distance=10pt,scale=0.7, white,opacity=0] (w4) {};
    \node[vertex,above of=v5,node distance=10pt,scale=0.7, white,opacity=0] (w5) {};
    \node[vertex,above of=v6,node distance=10pt,scale=0.7, white,opacity=0] (w6) {};
    \node[vertex,above of=v7,node distance=10pt,scale=0.7, white,opacity=0] (w7) {};
    \node[vertex,above of=v8,node distance=10pt,scale=0.7, white,opacity=0] (w8) {};
    \node[vertex,above of=v9,node distance=10pt,scale=0.7, white,opacity=0] (w9) {};
    \node[vertex,above of=v11,node distance=10pt,scale=0.7, white,opacity=0] (w11) {};
    \node[vertex,above of=v14,node distance=10pt,scale=0.7, white,opacity=0] (w14) {};
    \node[vertex,above of=v19,node distance=10pt,scale=0.7, white,opacity=0] (w19) {};

    \node[vertex,above of=v1,node distance=10pt,scale=0.7, white,opacity=0] (w1) {};
    \draw [->, red, line width=0.6pt] (w1) -- (v1);
    \node[vertex,above of=v4,node distance=10pt,scale=0.7, white,opacity=0] (w4) {};
    \node[vertex,above of=v6,node distance=10pt,scale=0.7, white,opacity=0] (w6) {};
    \node[vertex,above of=v8,node distance=10pt,scale=0.7, white,opacity=0] (w8) {};
    \node[vertex,above of=v11,node distance=10pt,scale=0.7, white,opacity=0] (w11) {};
    \node[vertex,above of=v20,node distance=10pt,scale=0.7, white,opacity=0] (w20) {};
    \node[vertex,above of=v5,node distance=10pt,scale=0.7, white,opacity=0] (w5) {};
    \node[vertex,above of=v10,node distance=10pt,scale=0.7, white,opacity=0] (w10) {};

    \draw [->, red, line width=0.6pt] (w1) -- (v1);
    \draw [->, red, line width=0.6pt] (w2) -- (v2);
    \draw [->, red, line width=0.6pt] (w4) -- (v4);
    \draw [->, red, line width=0.6pt] (w5) -- (v5);
    \draw [->, red, line width=0.6pt] (w6) -- (v6);
    \draw [->, red, line width=0.6pt] (w7) -- (v7);
    \draw [->, red, line width=0.6pt] (w8) -- (v8);
    \draw [->, red, line width=0.6pt] (w9) -- (v9);
    \draw [->, red, line width=0.6pt] (w11) -- (v11);
    \draw [->, red, line width=0.6pt] (w14) -- (v14);
    \draw [->, red, line width=0.6pt] (w20) -- (v20);
    
    \end{tikzpicture}
    \begin{tikzpicture}[scale=0.022]
    \node[circle,inner sep=0.8] at  (60,150) {\small\textbf{Re-fed}};
    \node[circle,draw,fill=black,color=black,inner sep=4.12311/4] (v1) at  (110,36) {};
    \node[circle,draw,fill=black,color=black,inner sep=4.24264/4] (v2) at  (13,78) {};
    \node[circle,draw,fill=black,color=black,inner sep=4.47214/4] (v3) at  (96,58) {};
    \node[circle,draw,fill=black,color=black,inner sep=5.65685/4] (v4) at  (85,102) {};
    \node[circle,draw,fill=black,color=black,inner sep=5.83095/4] (v5) at  (42,22) {};
    \node[circle,draw,fill=black,color=black,inner sep=4.47214/4] (v6) at  (47,34) {};
    \node[circle,draw,fill=black,color=black,inner sep=4.24264/4] (v7) at  (53,67) {};
    \node[circle,draw,fill=black,color=black,inner sep=4.12311/4] (v8) at  (85,89) {};
    \node[circle,draw,fill=black,color=black,inner sep=4.12311/4] (v9) at  (78,42) {};
    \node[circle,draw,fill=black,color=black,inner sep=4.12311/4] (v10) at  (53,2) {};
    \node[circle,draw,fill=black,color=black,inner sep=4.24264/4] (v11) at  (62,23) {};
    \node[circle,draw,fill=black,color=black,inner sep=4.47214/4] (v14) at  (71,88) {};
    \node[circle,draw,fill=black,color=black,inner sep=4.47214/4] (v16) at  (117,32) {};
    \node[circle,draw,fill=black,color=black,inner sep=5.83095/4] (v19) at  (118,27) {};
    \node[circle,draw,fill=black,color=black,inner sep=4.12311/4] (v20) at  (84,20) {};

    \node[vertex,above of=v1,node distance=10pt,scale=0.7, white,opacity=0] (w1) {};
    \draw [->, red, line width=0.6pt] (w1) -- (v1);
    \node[vertex,above of=v4,node distance=10pt,scale=0.7, white,opacity=0] (w4) {};
    \node[vertex,above of=v6,node distance=10pt,scale=0.7, white,opacity=0] (w6) {};
    \node[vertex,above of=v8,node distance=10pt,scale=0.7, white,opacity=0] (w8) {};
    \node[vertex,above of=v11,node distance=10pt,scale=0.7, white,opacity=0] (w11) {};
    \node[vertex,above of=v20,node distance=10pt,scale=0.7, white,opacity=0] (w20) {};
    \node[vertex,above of=v5,node distance=10pt,scale=0.7, white,opacity=0] (w5) {};
    \node[vertex,above of=v10,node distance=10pt,scale=0.7, white,opacity=0] (w10) {};

    \draw [->, red, line width=0.6pt] (w4) -- (v4);
    \draw [->, red, line width=0.6pt] (w2) -- (v2);
    \draw [->, red, line width=0.6pt] (w6) -- (v6);
    \draw [->, red, line width=0.6pt] (w8) -- (v8);
    \draw [->, red, line width=0.6pt] (w11) -- (v11);
    \draw [->, red, line width=0.6pt] (w20) -- (v20);
    
    \node[vertex,above of=v3,node distance=10pt,scale=0.7, white,opacity=0] (w3) {};
    \draw [->, red, line width=0.6pt] (w3) -- (v3);
    
    
    \draw [dashed, thick,gray] (v1) -- (v5);
    \draw [dashed, thick,gray] (v1) -- (v7);
    \draw [dashed, thick,gray] (v1) -- (v9);
    \draw [dashed, thick,gray] (v5) -- (v7);
    \draw [dashed, thick,gray] (v5) -- (v9);
    \draw [dashed, thick,gray] (v5) -- (v10);
    \draw [dashed, thick,gray] (v5) -- (v16);
    \draw [dashed, thick,gray] (v7) -- (v9);
    \draw [dashed, thick,gray] (v7) -- (v10);
    \draw [dashed, thick,gray] (v7) -- (v14);
    \draw [dashed, thick,gray] (v7) -- (v16);
    \draw [dashed, thick,gray] (v7) -- (v19);
    \draw [dashed, thick,gray] (v9) -- (v10);    
    \draw [dashed, thick,gray] (v9) -- (v16); 
    \draw [dashed, thick,gray] (v10) -- (v14); 
    \draw [dashed, thick,gray] (v9) -- (v19); 
    \draw [dashed, thick,gray] (v9) -- (v14); 
    \draw [dashed, thick,gray] (v10) -- (v19); 
    \end{tikzpicture}
    }
     
    \vspace{0.2cm}
    \begin{tikzpicture}[scale=0.73][font=\Large]
    \begin{axis}[ybar, enlargelimits=0.23, ymax =15, ymin=3.2,
    ylabel={\textbf{MON}}, 
    symbolic x coords={A,B,C},
    xtick=data, ylabel near ticks, xticklabels={\textbf{Fed},\textbf{Fast},\textbf{Re-fed}}, title=\textbf{D}, xlabel=\textbf{Phase}, legend pos = north east,legend style={draw=none}, xtick pos=left,ytick pos=left,legend columns=-1,
    ]
\addplot+[blue, blue, fill=blue, opacity=.7] coordinates {(A,8) (B,12) (C,8)};
\addlegendentry{\textbf{Hypergraphs}}

\addplot+[orange, red, fill=red, opacity=.7] coordinates {(A,15) (B,15) (C,15)};
\addlegendentry{\textbf{Graphs}}

\end{axis}
\end{tikzpicture}
\hspace{1cm}
    \caption{(Above) Mouse neuron endomicroscopy features. Neuronal activity networks of the three phases - fed, fast and re-fed, which depicts the spatial location and size of individual cells. Each 2-simplex (i.e., a triangle) represents a hyperedge, and red arrows indicate nodes selected in MON. (Below) MON for the neuronal activity networks modelled by 3-uniform hypergraphs and standard graphs.}
    \label{fig:mouse}
\end{figure} 